\documentclass[a4paper,12pt, twoside, reqno]{amsart}
\usepackage{amssymb}
\usepackage{amsmath}
\newtheorem{theorem}{Theorem}

\newtheorem{defn}{Definition}
\newtheorem{ex}{Example}
\newtheorem{remark}{Remark}[section]

\title{Fixed point theorems for Kannan type mappings with applications to split feasibility  and variational inequality problems}
\author{Vasile Berinde$^{1,2}$}
\author{M\u ad\u alina P\u acurar$^{3}$}
\begin{document}
\maketitle \pagestyle{myheadings} \markboth{Vasile Berinde and M\u ad\u alina P\u acurar} {Fixed point theorems for Kannan type mappings...}
\begin{abstract}
 The aim of this paper in to introduce a large class of mappings, called {\it enriched Kannan mappings}, that includes all Kannan mappings and some nonexpansive mappings. We study the set of fixed points and prove a convergence theorem for Kransnoselskij iteration used to approximate fixed points of enriched Kannan mappings in Banach spaces. We then extend further these mappings to the class of enriched Bianchini mappings. Examples to illustrate the effectiveness of our results are also given. As applications of our main fixed point theorems, we present two Kransnoselskij projection type algorithms for solving split feasibility problems and variational inequality problems in the class of enriched Kannan mappings and enriched Bianchini mappings, respectively.

\end{abstract}
\section{Introduction}
Let $X$ be a nonempty set and $T:X\rightarrow X$ a self mapping. We denote the set of fixed points of $T$ by $Fix\,(T)$, i.e., $Fix\,(T)=\{a\in X: T(a)=a\}$ and define the $n^{th}$ iterate of $T$ as usually, that is, $T^0=I$ (identity map) and $T^n=T^{n-1}\circ T$, for $n\geq 1$.

The mapping $T$ is said to be a {\it Picard operator}, see for example Rus \cite{Rus01}, if:\\ (i) $Fix\,(T)=\{p\}$  \quad and (ii)  $T^n(x_0) \rightarrow p$ as $n\rightarrow \infty$, for any $x_0$ in $X$.

The most important and useful class of Picard operators, that play a crucial role in nonlinear analysis, is the  class of mappings generally known in literature as {\it Banach contractions}, first introduced by Banach in \cite{Ban22}, in the case of what we call now a Banach space, and then extended to the setting of complete metric spaces by Caccioppoli  \cite{Cac}. Let $(X,d)$ be a metric space. A mapping $T:X\rightarrow X$ is called a {\it Banach contraction} if there exists a constant $c\in[0,1)$ such that
\begin{equation} \label{eq1}
d(Tx,Ty)\leq c\cdot d(x,y),\forall x,y \in X.
\end{equation}

Banach contraction mapping principle essentially states that, in a complete metric space $(X,d)$, any Banach contraction $T:X\rightarrow X$
is a Picard operator.

It is obvious by \eqref{eq1} that any Banach contraction is continuous but, in general, a Picard operator is not necessarily continuous. The first example of such discontinuous Picard operators has been given by Kannan in 1968, see \cite{Kan68}, \cite{Kan69}. A mapping $T:X\rightarrow X$ is called a {\it Kannan mapping} if there exists a constant $a\in[0,1/2)$ such that
\begin{equation} \label{eq2}
d(Tx,Ty)\leq a \left[d(x,Tx)+d(y,Ty)\right],\forall x,y \in X,
\end{equation}

Kannan \cite{Kan68} proved the following fixed point theorem, see also \cite{Kan69}.
\begin{theorem}\label{th0}
Let $(X,d)$ be  a complete metric space and let $T:X\rightarrow X$ be a Kannan mapping. Then $T$ is a Picard operator.
\end{theorem} 

It is important to note that the class of Kannan mappings \eqref{eq2} is independent of that of Banach contractions \eqref{eq1}, see \cite{Rho}, \cite{Mes}  and \cite{Ber04}.

Banach and Kannan contractions also exhibit a different behaviour with respect to the completeness of the ambient space: while Kannan contraction mapping principle characterises the metric completeness, as shown by Subrahmanyam \cite{Sub}, Banach contraction mapping principle does not, see  \cite{Con}, where it is presented an example of a metric space $X$ such that $X$ is not complete and every contraction on $X$ has a fixed point. There are many other research works on the similarities and diferrences between Banach contractions and Kannan mappings, see \cite{Ber04}, \cite{Ber07}, \cite{Del09}, \cite{Del10}, \cite{Kik08}, \cite{Kik08a}, \cite{Pac09}, \cite{Shi}, \cite{Suz05} and \cite{Wlo}. 

These are only some of the many reasons why Kannan contractions and generalizations of Kannan mappings play a particularly important role in fixed point theory and nonlinear analysis and attracted a rather important research work in the last decades, see \cite{Bal}, \cite{Ber07}, \cite{BPet}, \cite{Bla}, \cite{Dom}, \cite{Gor17}, \cite{Gor18}, \cite{Mar}, \cite{Nak}, \cite{Pac09}, \cite{Reich71}-\cite{Reich72}, \cite{Rus79}-\cite{Ume} and \cite{Wlo} and references therein.

In this paper, by using the technique of enrichment the contractive type mappings by Krasnoselskij averaging process, previously used in the papers \cite{Ber19}, \cite{Ber19a}, \cite{Pac19b}, \cite{Pac19c},  we introduce a class of generalized Kannan mappings and prove corresponding fixed point theorems for such contractions in Banach spaces. 

Examples to illustrate the richness of the new classes of contractions as well as some relevant applications to solving split feasibility problems and variational inequality problems are also given.

\section{Approximating fixed points of enriched Kannan mappings} 

\begin{defn}
Let $(X,\|\cdot\|)$ be a linear normed space. A mapping $T:X\rightarrow X$ is said to be an {\it enriched Kannan mapping} if there exist  $a\in[0,1/2)$ and $k\in[0,ü\infty)$ such that
\begin{equation} \label{eq3}
\|k(x-y)+Tx-Ty\|\leq a \left[\|x-Tx\|+\|y-Ty\|\right],\forall x,y \in X.
\end{equation}
To indicate the constants involved in \eqref{eq3} we shall also call  it  $(k,a$)-{\it enriched Kannan mapping}. 
\end{defn}

\begin{ex}[ ] \label{ex1}
\indent

(1) Any Kannan mapping is a $(0,a)$-enriched Kannan mapping, i.e., it satisfies \eqref{eq3} with $k=0$.

(2) Let $X=[0,1]$ be endowed with the usual norm and $T:X\rightarrow X$ be defined by $Tx=1-x$, for all $x\in [0,1]$. It is easy to check that $T$ is nonexpansive (it is an isometry),  $T$ is not a Kannan mapping but $T$ is an enriched Kannan mapping. Indeed, if $T$ would be a Kannan mapping, then there would exist $a\in[0,1/2)$ such that 
$$
|x-y|\leq a \cdot \left[|2x-1|+|2y-1|\right],\forall x,y \in [0,1],
$$
which, for $x=1/2$ and $y=1$, yields the contradiction $1\leq 2 a< 1$. 

The enriched Kannan mapping condition \eqref{eq3} is in this case equivalent to
$$
|(k-1)(x-y)|\leq a \cdot \left[|2x-1|+|2y-1|\right],\forall x,y \in [0,1],
$$
which, in view of the fact that $2|x-y|\leq |2x-1|+|2y-1|$, holds true if one chooses $|k-1|=2 a$, for any $a\in[0,1/2)$. This is possible only for $k<1$, which yields $k=1-2a>0$. Hence, for any $a\in[0,1/2)$, $T$ is a $(1-2a,a)$-enriched Kannan mapping and $Fix\,(T)=\left\{\dfrac{1}{2}\right\}$.

(3) Any of the mappings $T$ in Examples 1.3.1, 1.3.2 and 1.3.5 in P\u acurar \cite{Pac09} are discontinuous enriched Kannan mappings.

\end{ex}
\begin{remark}
We note that for $T$ in Example \ref{ex1},  (2), Picard iteration $\{x_n\}$ associated to $T$, that is, $x_{n+1}=1-x_n$, $n\geq 0$, does not converge for any $x_0$ different of $\dfrac{1}{2}$, the unique fixed point to $T$. 

This suggest us that in order to approximate fixed points of enriched Kannan mappings we need more elaborate fixed point iterative schemes.

The next result provides a convergence theorem for the Krasnoselskij iterative method in the class of {\it enriched} Kannan mappings.
\end{remark}

\begin{theorem}  \label{th1}
Let $(X,\|\cdot\|)$ be a Banach space and $T:X\rightarrow X$ a $(k,a$)-{\it enriched Kannan mapping}. Then

$(i)$ $Fix\,(T)=\{p\}$;

$(ii)$ There exists $\lambda\in (0,1]$ such that the iterative method
$\{x_n\}^\infty_{n=0}$, given by
\begin{equation} \label{eq3a}
x_{n+1}=(1-\lambda)x_n+\lambda T x_n,\,n\geq 0,
\end{equation}
converges to p, for any $x_0\in X$;

$(iii)$ The following estimate holds
\begin{equation}  \label{3.2-1}
\|x_{n+i-1}-p\| \leq\frac{\delta^i}{1-\delta}\cdot \|x_n-
x_{n-1}\|\,,\quad n=0,1,2,\dots;\,i=1,2,\dots
\end{equation}
where $\delta=\dfrac{a}{1-a}$.
\end{theorem}

\begin{proof}
For any $\lambda\in (0,1)$ consider the averaged mapping $T_\lambda$, given by
\begin{equation} \label{eq4}
T_\lambda (x)=(1-\lambda)x+\lambda T(x), \forall  x \in X.
\end{equation}
It easy to prove that $T_\lambda$ possesses the following important property:
$$
Fix(\,T_\lambda)=Fix\,(T).
$$
If $k> 0$ in \eqref{eq3}, then let us put  $\lambda=\dfrac{1}{k+1}$. Obviously, we have $0<\lambda<1$ and thus the contractive condition \eqref{eq3} becomes
$$
\left \|\left(\frac{1}{\lambda}-1\right)(x-y)+Tx-Ty\right\|\leq a \left[\|x-Tx\|+\|y-Ty\|\right],\forall x,y \in X,
$$
which can be written in an equivalent form as
\begin{equation} \label{eq5}
\|T_\lambda (x)-T_\lambda y\|\leq a  \left[\|x-T_\lambda x\|+\|y-T_\lambda y\|\right],\forall x,y \in X.
\end{equation}
The above inequality shows that $T_\lambda$ is a Kannan mapping. 

According to \eqref{eq4}, the iterative process $\{x_n\}^\infty_{n=0}$ defined by  \eqref{eq3a} is the Picard iteration associated to $T_\lambda$, that is,
$$
x_{n+1}=T_\lambda x_n,\,n\geq 0.
$$ 
Take $x=x_n$ and $y=x_{n-1}$ in  \eqref{eq5} to get
$$
\|x_{n+1}-x_{n}\|\leq a \left(\|x_n-x_{n+1}\|+\|x_{n}-x_{n-1}\|\right),
$$
which yields
$$
\|x_{n+1}-x_{n}\|\leq \frac{a}{1-a}\|x_{n}-x_{n-1}\|,\,n\geq 1.
$$
Since $0<a<\dfrac{1}{2}$, by denoting $\delta=\dfrac{a}{1-a}$, we have $0<\delta<1$ and therefore the sequence $\{x_n\}^\infty_{n=0}$ satisfies
\begin{equation} \label{eq6}
\|x_{n+1}-x_{n}\|\leq \delta \|x_{n}-x_{n-1}\|,\,n\geq 1.
\end{equation}
By \eqref{eq6} one obtains routinely the following two estimates
\begin{equation} \label{eq7a}
\|x_{n+m}-x_{n}\|\leq \delta^n \cdot \frac{1-\delta^m}{1-\delta}\cdot \|x_{1}-x_{0}\|,\,n\geq 0, m\geq 1.
\end{equation}
and
\begin{equation} \label{eq8}
\|x_{n+m}-x_{n}\|\leq \delta \cdot \frac{1-\delta^m}{1-\delta}\cdot \|x_{n}-x_{n-1}\|,\,n\geq 1, \,m\geq 1.
\end{equation}
Now, by \eqref{eq7a} it follows that $\{x_n\}^\infty_{n=0}$ is a Cauchy sequence and hence it is convergent in the Banach space $(X,\|\cdot\|)$. Let us denote
\begin{equation} \label{eq10a}
p=\lim_{n\rightarrow \infty} x_n.
\end{equation}
We first prove that $p$ is a fixed point of $T_\lambda$. We have
\begin{equation} \label{eq9}
\|p-T_\lambda p\|\leq \|p-x_{n+1}\|+\|x_{n+1}-T_\lambda p\|=\|x_{n+1}-p\|+\|T_\lambda x_{n}-T_\lambda p\|.
\end{equation}
By \eqref{eq5} it results that 
$$
\|T_\lambda x_{n}-T_\lambda p\|\leq a \left[\|x_{n}-T_\lambda x_{n}\|+\|p-T_\lambda p\|\right],
$$
and therefore, by \eqref{eq9} we obtain 
\begin{equation} \label{eq10}
\|p-T_\lambda p\|\leq \frac{1}{1-a} \|x_{n+1}-p\|+\delta \|x_{n+1}-x_{n}\|,\,n\geq 0.
\end{equation}
Now, by letting  $n\rightarrow \infty$ in \eqref{eq10} we get $\|p-T_\lambda p\|=0$, that is, $p=T_\lambda p$. So, $p\in Fix\,(T_\lambda)$. 

We prove that $p$ is the unique fixed point of $T_\lambda$. Assume that $q\neq p$ is another fixed point of $T_\lambda$. Then, by \eqref{eq5} 
$$
0<\|p-q\|\leq a \cdot 0,
$$
a contradiction. Hence $Fix\,(T_\lambda)=\{p\}$ and since $Fix\,(T)=Fix(\,T_\lambda)$, claim $(i)$ is proven. 

Conclusion $(ii)$ now follows by \eqref{eq10a}.

To prove $(iii)$, we let $m\rightarrow \infty$ in \eqref{eq7a}  and \eqref{eq8} to get
\begin{equation} \label{eq11}
\|x_n-p\|\leq  \frac{\delta^n}{1-\delta}\cdot \|x_{1}-x_{0}\|,\,n\geq 1
\end{equation}
and
\begin{equation} \label{eq12}
\|x_n-p\|\leq \frac{\delta}{1-\delta}\cdot \|x_{n}-x_{n-1}\|,\,n\geq 1,
\end{equation}
respectively and then  we merge \eqref{eq11}  and \eqref{eq12} to get the unifying error estimate \eqref{3.2-1}.

The remaining case $k=0$ is similar to $k\neq 0$ with the only difference that in this case $\lambda=1$ and hence we work with $T(=T_1)$, when Kasnoselskij iteration \eqref{eq3a} reduces to Picard iteration
$$
x_{n+1}=T x_n,\,n\geq 0.
$$
\end{proof}

\begin{remark}
1) It is well known, see for example Berinde \cite{Ber07}, that any Kannan mapping is a strictly quasi contractive mapping, that is
$$
\|Tx-p\|\leq \theta \cdot \|x-p\|,\forall x \in X, p\in Fix\,(T),\,(0<\theta <1).
$$
On the other hand, $T$ in Example \ref{ex1} (2) is nonexpansive and is not strictly quasi-contractive. So, enriched Kannan mapping is a larger class of mappings than the class of strictly quasi-contractive mappings.

2) In the particular case $k=0$, by Theorem \ref{th1} we get the classical Kannan fixed point theorem (see Kannan \cite{Kan68}) in the setting of a Banach space.

3) As proved by Subrahmanyam \cite{Sub}, Kannan's fixed point theorem characterizes the metric completeness. That is, a metric space $(X,d)$ is complete if and only if every Kannan mapping on $X$  has a fixed point. So, it is an open question whether or not enriched Kannan mapping mapping principle (Theorem \ref{th1}) still characterizes the metric completeness of the ambient space.

\end{remark}

\section{Approximating fixed points of enriched Bianchini mappings} 

In the renown Rhoades' classification of contractive conditions \cite{Rho}, Banach contraction condition is numbered (1),  while  Kannan mapping condition is numbered (4), followed by Bianchini's contraction condition, numbered (5), which has been introduced and studied in \cite{Bia}. As any Kannan mapping is a Bianchini mapping but the reverse is generally not true (just use the function $T:[0,1]\rightarrow [0,1]$), $Tx=x/3$, $x\in [0,1]$, to prove this assertion), our aim in this section is to introduce and study the class of enriched Bianchini mappings which is independent of the class of enriched Kannan mappings.

\begin{defn} \label{def2}
Let $(X,\|\cdot\|)$ be a linear normed space. A mapping $T:X\rightarrow X$ is said to be an {\it enriched Bianchini mapping} if there exist  $h\in[0,1)$ and $k\in[0,ü\infty)$ such that
\begin{equation} \label{eq13}
\|k(x-y)+Tx-Ty\|\leq h \max\left\{\|x-Tx\|,\|y-Ty\|\right\},\forall x,y \in X.
\end{equation}
To indicate the constants involved in \eqref{eq13} we shall call  $T$ a $(k,h$)-{\it enriched Bianchini mapping}. 
\end{defn}

\begin{ex}[ ] \label{ex2}
\indent

(1) If $k=0$ then by \eqref{eq13}, we obtain the original Bianchini contraction condition. So, any Bianchini mapping is a $(0,h$)-enriched Bianchini mapping.

(2) Any $(k,a$)-enriched Kannan mapping is a $(k,h$)-enriched Bianchini mapping, with $h=2 a$, in view of the inequality $u+v\leq 2 \max\{u,v\}$. This implies that the nonexpansive map given in Example \ref{ex1} (2) is a $(2(1-a),2a)$-enriched Bianchini mapping, for any $a\in (0,1/2)$. For this mapping, as shown in the previous section, Picard iteration does not converge, in general. 

\end{ex}
The next result proves that fixed points  of strictly {\it enriched} Bianchini mappings can be approximated by  Krasnoselskij iterative method.

\begin{theorem}  \label{th2}
Let $(X,\|\cdot\|)$ be a Banach space and $T:X\rightarrow X$ a $(k,h$)-{\it enriched Bianchini mapping}. Then

$(i)$ $Fix\,(T)=\{p\}$;

$(ii)$ There exists $\lambda\in (0,1]$ such that the iterative method
$\{x_n\}^\infty_{n=0}$, given by
\begin{equation} \label{eq1.3a}
x_{n+1}=(1-\lambda)x_n+\lambda T x_n,\,n\geq 0,
\end{equation}
converges to p, for any $x_0\in X$;

$(iii)$ The following estimate holds
\begin{equation}  \label{1.3b}
\|x_{n+i-1}-p\| \leq\frac{h^i}{1-h}\cdot \|x_n-
x_{n-1}\|\,,\quad n=0,1,2,\dots;\,i=1,2,\dots
\end{equation}
\end{theorem}

\begin{proof}
For any $\lambda\in (0,1)$, consider the averaged mapping $T_\lambda$, given by
\begin{equation} \label{eq1.4}
T_\lambda (x)=(1-\lambda)x+\lambda T(x), \forall  x \in X,
\end{equation}
which has the  property $
Fix(\,T_\lambda)=Fix\,(T).
$

If $k> 0$ in \eqref{eq13}, then let us denote  $\lambda=\dfrac{1}{k+1}\in (0,1)$. Thus the contractive condition \eqref{eq1.3a} becomes
$$
\left \|\left(\frac{1}{\lambda}-1\right)(x-y)+Tx-Ty\right\|\leq h \max\left\{\|x-Tx\|,\|y-Ty\|\right\},\forall x,y \in X,
$$
which can be written in an equivalent form as
\begin{equation} \label{eq1.5}
\|T_\lambda x-T_\lambda y\|\leq h  \max\left\{\|x-T_\lambda x\|,\|y-T_\lambda y\|\right\},\forall x,y \in X.
\end{equation}
The above inequality shows that $T_\lambda$ is a Bianchini mapping. 

According to \eqref{eq1.4}, the iterative process $\{x_n\}^\infty_{n=0}$ defined by  \eqref{eq1.3a} is the Picard iteration associated to $T_\lambda$, that is,
$$
x_{n+1}=T_\lambda x_n,\,n\geq 0.
$$ 
Take $x=x_n$ and $y=x_{n-1}$ in  \eqref{eq1.5} to get
\begin{equation} \label{eq1.6}
	\|x_{n+1}-x_{n}\|\leq h \max\left\{\|x_n-x_{n+1}\|,\|x_{n}-x_{n-1}\|\right).
\end{equation}
There are two possible cases.

{\bf Case 1.} $\max\left\{\|x_n-x_{n+1}\|,\|x_{n}-x_{n-1}\|\right)=\|x_n-x_{n+1}\|$. In this case,  by \eqref{eq1.6} it results
$$
\|x_{n+1}-x_{n}\|\leq h \|x_n-x_{n+1}\|<\|x_n-x_{n+1}\|,
$$
a contradiction. Therefore, the only possibility is  

{\bf Case 2.} $\max\left\{\|x_n-x_{n+1}\|,\|x_{n}-x_{n-1}\|\right)=\|x_{n}-x_{n-1}\|$, when by  \eqref{eq1.6} it results that the sequence $\{x_n\}^\infty_{n=0}$ satisfies the inequality
\begin{equation} \label{eq6.1}
\|x_{n+1}-x_{n}\|\leq h \|x_{n}-x_{n-1}\|,\,n\geq 1.
\end{equation}
By \eqref{eq6.1} one obtains routinely the following two estimates
\begin{equation} \label{eq1.7a}
\|x_{n+m}-x_{n}\|\leq h^n \cdot \frac{1-h^m}{1-h}\cdot \|x_{1}-x_{0}\|,\,n\geq 0, m\geq 1,
\end{equation}
and
\begin{equation} \label{eq1.8}
\|x_{n+m}-x_{n}\|\leq h \cdot \frac{1-h^m}{1-h}\cdot \|x_{n}-x_{n-1}\|,\,n\geq 1, \,m\geq 1.
\end{equation}
Now, by \eqref{eq1.7a} it follows that $\{x_n\}^\infty_{n=0}$ is a Cauchy sequence and hence it is convergent in the Banach space $(X,\|\cdot\|)$. Let us denote
\begin{equation} \label{eq1.10a}
p=\lim_{n\rightarrow \infty} x_n.
\end{equation}
We prove that $p$ is a fixed point of $T_\lambda$. We have
\begin{equation} \label{eq1.9}
\|p-T_\lambda p\|\leq \|p-x_{n+1}\|+\|x_{n+1}-T_\lambda p\|=\|x_{n+1}-p\|+\|T_\lambda x_{n}-T_\lambda p\|,
\end{equation}
and so by \eqref{eq1.5} it results that 
$$
\|T_\lambda x_{n}-T_\lambda p\|\leq h \max\left\{\|x_n-x_{n+1}\|,\|p-T_\lambda p\|\right\}.
$$
Now, if $\max\left\{\|x_n-x_{n+1}\|,\|p-T_\lambda p\|\right\}=\|x_n-x_{n+1}\|$, then we get
$$
\|T_\lambda x_{n}-T_\lambda p\|\leq h \|x_n-x_{n+1}\|\leq \dots \leq h^n \|x_0-x_{1}\|,
$$
and by \eqref{eq1.9} one obtains
$$
\|p-T_\lambda p\|\leq \|x_{n+1}-p\|+h^n \|x_0-x_{1}\|,\,n\geq 0,
$$
from which by letting $n\rightarrow \infty$ we get $\|p-T_\lambda p\|=0$, that is, $p=T_\lambda p$ and so, $p\in Fix\,(T_\lambda)$.

If $\max\left\{\|x_n-x_{n+1}\|,\|p-T_\lambda p\|\right\}=\|p-T_\lambda p\|$,\\ then by \eqref{eq1.9} one obtains 
\begin{equation} \label{eq1.10}
\|p-T_\lambda p\|\leq \frac{1}{1-h} \|x_{n+1}-p\|,\,n\geq 0,
\end{equation}
from which, by letting  $n\rightarrow \infty$ we get $\|p-T_\lambda p\|=0$, that is,  $p\in Fix\,(T_\lambda)$. 

We now prove that $p$ is the unique fixed point of $T_\lambda$. Assume that $q\neq p$ is another fixed point of $T_\lambda$. Then, by \eqref{eq1.5} 
$$
0<\|p-q\|\leq a \cdot 0,
$$
a contradiction. Hence $Fix\,(T_\lambda)=\{p\}$ and since $Fix\,(T)=Fix(\,T_\lambda)$, claim $(i)$ is proven. 

Conclusions $(ii)$ and $(iii)$  follow similarly to the proof of Theorem \ref{th1}.

The remaining case $k=0$ is also similar to $k\neq 0$ with the only difference that now $\lambda=1$ and hence we shall work with $T=T_1$, when Kasnoselskij iteration \eqref{eq3a} reduces to the simple Picard iteration
$$
x_{n+1}=T x_n,\,n\geq 0.
$$
\end{proof}

\section{Applications to split feasibility problems and variational inequality problems}

A large class of nonlinear problems can be solved by finding a solution of an equivalent fixed point problem $x=Tx$.  For such kind of problems, the iterative algorithms considered are selecting a member of the set $Fix\,(T )$ which is a solution to the original  problem. In many concrete applications, like the ones occurring in signal processing or image reconstruction problems, the operator $T$ is nonexpansive, when the technique of averaging mappings is extensively used, see Byrne \cite{Byr}. 

As we have seen in the previous two sections of the present paper, it is very useful to consider the averaging technique not necessarily for nonexpansive mappings but for other classes of contractions that include some nonexpansive mappings.

The aim of this section is to present generic convergence theorems for Krasnoselskij type algorithms that solve split feasibility problems and variational inequality problems, respectively. 

\subsection{Solving split feasibility problems}

The spit feasibility problem (SFP), introduced by Censor and Elfving in 1994 \cite{Cen} is: 
\begin{equation} \label{a1}
\text{ Find } x^*\in C \text{ such that } Ax^*\in Q,
\end{equation}
where $C$ and $Q$ are closed convex subsets of the Hilbert spaces $H_1$ and $H_2$, respectively, and $A: H_1\rightarrow H_2$ is a bounded linear operator. Various concrete problems in signal processing, like phase retrieval and design of a nonlinear synthetic discriminant filter for optical pattern recognition, see \cite{Xu}, can by formulated as SFPs.

If we assume that the SFP \eqref{a1} is consistent, that is, it has a solution and denote by $S$ the solution set of \eqref{a1}, then, see \cite{Xu}, $x^*\in C$ is a solution of \eqref{a1} if and only if it is a solution of the fixed point problem
\begin{equation} \label{a2}
x=P_C\left(I-\gamma A^*\left(I-P_Q\right)A\right)x,
\end{equation}
where $P_C$ and $P_Q$ are the nearest point projections onto $C$ and $Q$, respectively, $\gamma >0$ and $A^*$ is the adjoint operator of $A$. It has been shown by Byrne \cite{Byr} that if $\delta$ is the spectral radius of $A^*A$ and $\gamma\in (0,2/\delta)$, then the operator
$$
T=P_C\left(I-\gamma A^*\left(I-P_Q\right)A\right)
$$
is averaged and nonexpansive and the so called CQ algorithm
\begin{equation} \label{a3}
x_{n+1}=P_C\left(I-\gamma A^*\left(I-P_Q\right)A\right)x_n,\,n\geq 0,
\end{equation}
converges {\it weakly} to a solution of the SFP.

In the case of averaged nonexpansive mappings, the problem of turning the weak convergence above into strong convergence received a great deal of research work. This usually consists in considering additional assumptions, see \cite{Xu} for a recent survey on Halpern type algorithms.

We propose here an  alternative to all those approaches, by considering enriched Kannan and enriched Bianchini mappings, respectively, which are in general discontinuous mappings, instead of nonexpansive mappings, which are always continuous. In this case we shall have a SFP with unique solution, as shown by the next theorem, while the the considered algorithm \eqref{a4} will convergence strongly.
\begin{theorem} \label{th3}
Assume that the SFP problem \eqref{a1} is consistent,  $\gamma\in (0,2/\delta)$ and $T=P_C\left(I-\gamma A^*\left(I-P_Q\right)A\right)$ is a $(k,a)$-enriched Kannan mapping. Then there exist $\lambda\in[0,1)$ such that the iterative algorithm $\{x_n\}$ defined by 
\begin{equation} \label{a4}
x_{n+1}=(1-\lambda)x_n+\lambda P_C\left(I-\gamma A^*\left(I-P_Q\right)A\right)x_n,\,n\geq 0,
\end{equation}
converges strongly to the unique solution $x^*$ of the SFP problem \eqref{a1}, for any $x_0\in C.$
\end{theorem} 

\begin{proof}
We apply Theorem \ref{th1} for $X=C$ and $T=P_C\left(I-\gamma A^*\left(I-P_Q\right)A\right)$.
\end{proof}

\subsection{Solving variational inequality problems}

Let $H$ be a Hilbert space and let $C\subset H$ be closed and convex. A mapping $G:H\rightarrow H$ is called {\it monotone} if
$$
\langle Gx-Gy,x-y\rangle \geq 0, \forall x,y\in H.
$$
The {\it variational inequality problem} with respect to $G$ and $C$, denoted by $VIP(G,C)$, is to find $x^*\in C$ such that
$$
\langle Gx^*,x-x^*\rangle \geq 0, \forall x\in H.
$$
It is well known, see for example \cite{Byr}, that if $\gamma>0$ then $x^*\in C$ is a solution of the $VIP(G,C)$ if and only if $x^*$ is a solution of the fixed point problem
\begin{equation} \label{a5}
x=P_C\left(I-\gamma G\right) x,
\end{equation}
where $P_C$ is the nearest point projection onto $C$. 

In Byrne \cite{Byr} it is proven, amongst many other important results, that if $I-\gamma G$ and $P_C\left(I-\gamma G\right)$ are averaged nonexpansive mappings, then, under some additional assumptions, the iterative algorithm $\{x_n\}$ defined by 
\begin{equation} \label{a6}
x_{n+1}=P_C\left(I-\gamma G\right) x_n,\,n\geq 0,
\end{equation}
converges {\it weakly} to a solution of the $VIP(G,C)$, if such solutions exist.

Our alternative is to consider $VIP(G,C)$ for enriched Bianchini mappings, which are in general discontinuous mappings, instead of nonexpansive mappings, which are always continuous. In this case we shall have $VIP(G,C)$  with a unique solution, as shown by the next theorem. Moreover, the considered algorithm \eqref{a7} will convergence strongly to the solution of the $VIP(G,C)$.

\begin{theorem} \label{th4}
Assume that for $\gamma >0$, $T=P_C\left(I-\gamma G\right)$ is a $(k,a)$-enriched Bianchini mapping. Then there exist $\lambda\in[0,1)$ such that the iterative algorithm $\{x_n\}$ defined by 
\begin{equation} \label{a7}
x_{n+1}=(1-\lambda)x_n+\lambda P_C\left(I-\gamma G\right)x_n,\,n\geq 0,
\end{equation}
converges strongly to the unique solution $x^*$ of the $VIP(G,C)$, for any $x_0\in C.$
\end{theorem} 

\begin{proof}
We apply Theorem \ref{th2} for $X=C$ and $T=P_C\left(I-\gamma G\right)$.
\end{proof}

\section{Conclusions and further developments}

In this paper we first introduced a large class of contractive mappings, called {\it enriched Kannan mappings}, that includes usual Kannan mappings and also includes some nonexpansive mappings. We studied the set of fixed points and constructed an algorithm of Kransnoselskij type in order to approximate fixed points of enriched Kannan mappings for which we have proved a strong convergence theorem. 

We then extended the enriched Kannan mappings to the slightly larger class of {\it enriched Bianchini mappings} and constructed the corresponding algorithm of Kransnoselskij type to approximate the fixed points of enriched Bianchini mappings for which we have also proved a strong convergence theorem. 
Nontrivial examples to illustrate the effectiveness of our fixed point results are also given. 

As applications of our main results, we presented two Kransnoselskij projection type algorithms for solving split feasibility problems and variational inequality problems in the class of enriched Kannan mappings and enriched Bianchini mappings, respectively, thus improving the existence and weak convergence results  for split feasibility problems and variational inequality problems in \cite{Byr}  to existence and uniqueness as well as to strong convergence theorems.

\vskip 0.5 cm {\it $^{1}$ Department of Mathematics and Computer Science

North University Center at Baia Mare

Technical University of Cluj-Napoca 

Victoriei 76, 430122 Baia Mare ROMANIA

E-mail: vberinde@cunbm.utcluj.ro}

\vskip 0.5 cm {\it $^{2}$ Academy of Romanian Scientists  (www.aosr.ro)

E-mail: vasile.berinde@gmail.com}

\vskip 0.5 cm {\it $^{3}$ Department of Statistics, Analysis, Forecast and Mathematics

Faculty of Economics and Bussiness Administration 

Babe\c s-Bolyai University of Cluj-Napoca,  Cluj-Napoca ROMANIA

E-mail: madalina.pacurar@econ.ubbcluj.ro}


\begin{thebibliography}{99}

\bibitem{Alg} Alghamdi, M., Berinde, V., Shahzad, N., {\it Fixed point of multivalued nonself almost contractions}, J. Appl. Math. {\bf 2013}, Art. ID 621614, 6 pp.

\bibitem{Bal} Balog, L., Berinde, V., {\it Fixed point theorems for nonself Kannan type contractions in Banach spaces endowed with a graph}, Carpathian J. Math. {\bf 32} (2016), no. 3, 293--302.

\bibitem{Ban22} Banach, S., {\it Sur les op\' erations dans les ensembles abstraits et leurs applications aux \' equations int\' egrales}, Fund Math. {\bf 3}  (1922), 133--181.

\bibitem{Ber04} Berinde, V.,  {\it Approximating fixed points of weak contractions using the Picard iteration}, Nonlinear Anal. Forum {\bf 9} (2004), no. 1, 43--53.

\bibitem{Ber07}  Berinde, V., {\it Iterative approximation of fixed points}, Second edition. Lecture Notes in Mathematics, 1912. Springer, Berlin, 2007.

\bibitem{Ber12}  Berinde, V., {\it Approximating fixed points of implicit almost contractions}, Hacettepe J. Math. Stat. {\bf 41} (2012), no. 1, 93-102.

\bibitem{Ber15}  Berinde, V., Petric, M.-A., {\it Fixed point theorems for cyclic non-self single-valued almost contractions}, Carpathian J. Math. {\bf 31} (2015), no. 3, 289-296.

\bibitem{Ber19}  Berinde, V., {\it Approximating fixed points of enriched nonexpansive mappings by Krasnoselskij iteration in Hilbert spaces}, Carpathian J. Math. {\bf 35} (2019), no. 3, 277-288. 

\bibitem{Ber19a}  Berinde, V., {\it Approximating fixed points of enriched strictly pseudocontractive operators in Hilbert spaces}  (submitted)

\bibitem{BPet} Berinde, V., Petric, M., {\it Fixed point theorems for cyclic non-self single-valued almost contractions}, Carpathian J. Math. {\bf 31} (2015), no. 3, 289--296.

\bibitem{Pac19b} Berinde, V., P\u acurar, M., {\it Approximating fixed points of enriched contractions in Banach spaces} (submitted)

\bibitem{Pac19c} Berinde, V., P\u acurar, M., {\it Approximating fixed points of enriched Chatterjea type mappings} (submitted)

\bibitem{Bia} Bianchini, R. M. T., {\it Su un problema di S. Reich riguardante la teoria dei punti fissi}, Boll. Un. Mat. Ital. {\bf 5} (1972), 103--108.

\bibitem{Bla} De Blasi, F.S., {\it Fixed points for KannanÕs mappings in Hilbert spaces}. Boll. Un. Mat. Ital.  9 (1974), no.4, 818--823.

\bibitem{Boj} Bojor, F., {\it Fixed points of Bianchini mappings in metric spaces endowed with a graph}, Carpathian J. Math. {\bf 28} (2012), no. 2, 207--214. 

\bibitem{Byr} Byrne, C., {\it A unified treatment of some iterative algorithms in signal processing and image reconstruction}, Inverse Problems {\bf 20} (2004), no. 1, 103--120.

\bibitem{Cac} Caccioppoli, R., {\it Un teorema generale sull'esistenza di elementi uniti in una transformazione funzionale}, Rend Accad dei Lincei. {\bf 11}  (1930), 794--799.

\bibitem{Cen} Censor, Y., Elfving, T., {\it A multiprojection algorithm using Bregman projections in a product space}, Numer. Algorithms {\bf 8} (1994), no. 2-4, 221--239.

\bibitem{Con} Connell, E. H., {\it Properties of fixed point spaces}, Proc. Amer. Math. Soc. {\bf 10} (1959) 974--979.

\bibitem{Del09} De la Sen, M., {\it Some combined relations between contractive mappings, Kannan mappings, reasonable nonexpansive mappings, and $T$-stability}, Fixed Point Theory Appl. {\bf 2010}, Art. ID 815637, 25 pp.

\bibitem{Del10} De la Sen, M., {\it Linking contractive self-mappings and cyclic Meir-Keeler contractions with Kannan self-mappings}, Fixed Point Theory Appl. {\bf 2009}, Art. ID 572057, 23 pp.

\bibitem{Dom} Dominguez, T., Lorenzo, J., Gatica, I., {\it Some generalizations of Kannan's fixed point theorem in $K$-metric spaces}, Fixed Point Theory {\bf 13} (2012), no. 1, 73---83.

\bibitem{Enj} Enjouji, Y., Nakanishi, M., Suzuki, T., {\it A generalization of Kannan's fixed point theorem}, Fixed Point Theory Appl. {\bf 2009}, Art. ID 192872, 10 pp.

\bibitem{Gor17} G\' ornicki, J., {\it Fixed point theorems for Kannan type mappings}, J. Fixed Point Theory Appl. {\bf 19} (2017), no. 3, 2145--2152.

\bibitem{Gor18} G\' ornicki, J., {\it Various extensions of Kannan's fixed point theorem}, J. Fixed Point Theory Appl. {\bf 20} (2018), no. 1, Art. 20, 12 pp.

\bibitem{Hor18} Horvat-Marc, A., Balog, L., {\it Fixed point theorems for nonself-Bianchini type contractions in Banach spaces endowed with a graph}, Creat. Math. Inform. {\bf 27} (2018), no. 1, 37--48.

\bibitem{Kan68} Kannan, R, {\it Some results on fixed points}, Bull. Calcutta Math. Soc. {\bf 60}  (1968), 71--76.

\bibitem{Kan69} Kannan, R., {\it Some results on fixed points. II}, Amer. Math. Monthly {\bf 76} (1969), 405--408.

\bibitem{Kik08} Kikkawa, M., Suzuki, T., {\it Some similarity between contractions and Kannan mappings}, Fixed Point Theory Appl. {\bf 2008}, Art. ID 649749, 8 pp.

\bibitem{Kik08a} Kikkawa, M., Suzuki, T., {\it Some similarity between contractions and Kannan mappings. II}, Bull. Kyushu Inst. Technol. Pure Appl. Math. No. 55 (2008), 1--13. 

\bibitem{Xu} L\' opez, G., Mart\' in-M\' arquez, V., Xu, H.-K., {\it Halpern's iteration for nonexpansive mappings}. Nonlinear analysis and optimization I. Nonlinear analysis, 211--231, Contemp. Math., 513, Israel Math. Conf. Proc., Amer. Math. Soc., Providence, RI, 2010.

\bibitem {Mar} M\u aru\c ster, \c St., Rus, I. A., {\it Kannan contractions and strongly demicontractive mappings}, Creat. Math. Inform. {\bf 24} (2015), no. 2, 171--180.

\bibitem {Mes} Meszaros, J., \emph{A comparison of various definitions of contractive type
mappings}, Bull. Calcutta Math. Soc. {\bf 84}  (1992), no. 2, 167--194.

\bibitem{Nak} Nakanishi, M., Suzuki, T., {\it An observation on Kannan mappings}, Cent. Eur. J. Math. {\bf 8} (2010), no. 1, 170--178.

\bibitem{Pac09} P\u acurar, M., {\it Iterative methods for fixed point approximation}, Editura Risoprint, Cluj-Napoca, 2009.

\bibitem{Pac14} P\u acurar, M., Berinde, V., Borcut, M., Petric, M., {\it Triple fixed point theorems for mixed monotone Pre\v si\' c-Kannan and Pre\v si\' c-Chaterjea mappings in partially ordered metric spaces}, Creat. Math. Inform. {\bf 23} (2014), no. 2, 223--234.

\bibitem{Reich71} Reich, S., {\it Kannan's fixed point theorem}, Boll. Un. Mat. Ital. 4 (1971), no. 4, 1--11.

\bibitem{Reich71a} Reich, S., {\it Some remarks concerning contraction mappings}, Can. Math. Bull. 14 (1971), 121--124.

\bibitem{Reich72} Reich, S., {\it Fixed points of contractive functions}, Boll. Un. Mat. Ital. 5 (1972), no. 4, 26--42.

\bibitem{Rho} Rhoades, B. E., {\it A comparison of various definitions of contractive mappings}, Trans. Amer. Math. Soc. {\bf 226} (1977), 257--290.

\bibitem {Rus79} Rus, I.A., \emph{Metrical Fixed Point Theorems}, Univ. of Cluj-Napoca, 1979.

\bibitem {Rus79a} Rus, I.A., \emph{Principles and Applications of the Fixed Point Theory} (in
Romanian), Editura Dacia, Cluj-Napoca, 1979

\bibitem {Rus83} Rus, I.A., \emph{Generalized contractions}, Seminar on Fixed Point Theory {\bf 3}
(1983) 1--130.
\bibitem {Rus96} Rus, I.A., {\it Picard operator and applications}, Babe\c ss-Bolyai University, 1996.

\bibitem {Rus01a} Rus, I.A., \emph{Generalized Contractions and Applications}, Cluj University Press, Cluj-Napoca, 2001.

 
\bibitem{Rus01} Rus, I.A., {\it Weakly Picard operators and applications}, Semin. Fixed Point Theory Cluj-Napoca {\bf 2} (2001), 41--57.

\bibitem{Shi} Shioji, N., Suzuki, T., Takahashi, W., {\it Contractive mappings, Kannan mappings and metric completeness}, Proc. Amer. Math. Soc. {\bf 126} (1998), no. 10, 3117--3124.

\bibitem{Sol} Soliman, A. H., Imdad, M., Ahmadullah, M., {\it Fixed Point theorems for uniformly generalized Kannan type semigroup of self-mappings}, Creat. Math. Inform. {\bf 26} (2017), no. 2, 231--240.

\bibitem{Sub} Subrahmanyam, P. V., {\it Remarks on some fixed-point theorems related to Banach's contraction principle}, J. Mathematical and Physical Sci. {\bf 8} (1974), 445--457; errata, ibid. {\bf 9} (1975), 195.

\bibitem{Suz05} Suzuki, T., {\it Contractive mappings are Kannan mappings, and Kannan mappings are contractive mappings in some sense}, Comment. Math. (Prace Mat.) {\bf 45} (2005), no. 1, 45--58.

\bibitem{Suz09} Enjouji, Y., Nakanishi, M., Suzuki, T., {\it A generalization of Kannan's fixed point theorem}, Fixed Point Theory Appl. {\bf 2009}, Art. ID 182872, 10 pp.

\bibitem{Ume} Ume, J. S., {\it Fixed point theorems for Kannan-type maps},   Fixed Point Theory Appl. {\bf 2015}, 2015: 38, 13 pp.

\bibitem{Wlo} Wlodarczyk, K., Plebaniak, R., {\it Kannan-type contractions and fixed points in uniform spaces}, Fixed Point Theory Appl. {\bf 2011}, 2011:90, 24 pp.

\end{thebibliography}
\end{document}